\documentclass[10pt,reqno]{amsart}

\usepackage[a4paper,left=35mm,right=35mm,top=30mm,bottom=30mm,marginpar=25mm]{geometry}

\usepackage{amsmath}
\usepackage{amssymb}
\usepackage{amsthm}
\usepackage[latin1]{inputenc}
\usepackage{eurosym}
\usepackage[dvips]{graphics}
\usepackage{graphicx}
\usepackage{epsfig}
\usepackage{hyperref}
\usepackage{dsfont}

\usepackage[displaymath,mathlines]{lineno}

\allowdisplaybreaks

\usepackage{hyperref}

\usepackage{ifthen}

\newcommand{\Rr}{{\mathbb{R}}}

\newcommand{\Nn}{{\mathbb{N}}}

\newcommand{\Tt}{{\mathbb{T}}}

\newcommand{\Hh}{{\overline{H}}}

\newcommand{\Ll}{{\mathcal{L}}}

\newcommand{\Dd}{{\mathcal{D}}}

\newcommand{\mbfm}{\boldsymbol{m}}
\newcommand{\mbfu}{\boldsymbol{u}}

\newcommand{\epsi}{\varepsilon}
\def\d{{\rm d}}
\def\dx{{\rm d}x}
\def\dy{{\rm d}y}

\def\dt{{\rm d}t}
\def\leq{\leqslant}
\def\geq{\geqslant}

\newcommand{\RR}{\mathbb{R}}
\newcommand{\NN}{\mathbb{N}}

\numberwithin{equation}{section}

\newtheoremstyle{thmlemcorr}{10pt}{10pt}{\itshape}{}{\bfseries}{.}{10pt}{{\thmname{#1}\thmnumber{
#2}\thmnote{ (#3)}}}
\newtheoremstyle{thmlemcorr*}{10pt}{10pt}{\itshape}{}{\bfseries}{.}\newline{{\thmname{#1}\thmnumber{
#2}\thmnote{ (#3)}}}
\newtheoremstyle{defi}{10pt}{10pt}{\itshape}{}{\bfseries}{.}{10pt}{{\thmname{#1}\thmnumber{
#2}\thmnote{ (#3)}}}
\newtheoremstyle{remexample}{10pt}{10pt}{}{}{\bfseries}{.}{10pt}{{\thmname{#1}\thmnumber{
#2}\thmnote{ (#3)}}}
\newtheoremstyle{ass}{10pt}{10pt}{}{}{\bfseries}{.}{10pt}{{\thmname{#1}\thmnumber{
A#2}\thmnote{ (#3)}}}

\theoremstyle{thmlemcorr}
\newtheorem{theorem}{Theorem}
\numberwithin{theorem}{section}
\newtheorem{lemma}[theorem]{Lemma}

\newtheorem{proposition}[theorem]{Proposition}

\theoremstyle{thmlemcorr*}
\newtheorem{theorem*}{Theorem}
\newtheorem{lemma*}[theorem]{Lemma}
\newtheorem{corollary*}[theorem]{Corollary}
\newtheorem{proposition*}[theorem]{Proposition}
\newtheorem{problem*}[theorem]{Problem}
\newtheorem{conjecture*}[theorem]{Conjecture}

\theoremstyle{defi}
\newtheorem{definition}[theorem]{Definition}

\theoremstyle{remexample}
\newtheorem{remark}[theorem]{Remark}

\newtheorem{pro}[theorem]{Proposition}

\theoremstyle{ass}

\usepackage{color}

\begin{document}

\title{Two numerical approaches to stationary mean-field games}
\author{Noha Almulla, Rita Ferreira, and Diogo Gomes } 

\begin{abstract}
Here, we consider numerical methods for stationary mean-field games (MFG)
and investigate two classes of algorithms. The first 
one is a gradient-flow method based on 
the variational characterization of certain MFG. 
The second one uses monotonicity
properties of MFG. 
We illustrate our methods with various examples, including one-dimensional periodic MFG, congestion problems, and higher-dimensional models. 
\end{abstract}

\thanks{
The authors were partially supported by King Abdullah
University of Science and Technology baseline and start-up funds and by 
KAUST SRI, Center for Uncertainty Quantification in Computational Science and Engineering. }

\address{N. Almulla,  University of Dammam,
College of Science, King Faisal Road, Dammam
Saudi Arabia }
\address{\vskip-6mm R. Ferreira and D. Gomes,  King Abdullah University
of Science
and Technology (KAUST), CEMSE Division \&
KAUST SRI, Center for
Uncertainty Quantification  in Computational
Science and Engineering,  Thuwal 23955-6900,
Saudi Arabia. \vskip.5mm
}

\email[N.~Almulla]{nalmulla@uod.edu.sa}
\email[R.~Ferreira]{rita.ferreira@kaust.edu.sa}
\email[D.~Gomes]{diogo.gomes@kaust.edu.sa}

\maketitle

\section{Introduction}

Mean-field games (MFG) model problems with a large number of rational agents interacting non-cooperatively \cite{ll1, ll2, ll3, Caines1, Caines2}.
Much progress has been achieved in the mathematical theory of MFG  for time-dependent problems \cite{porretta, porretta2, cgbt, GPM1, GPM2, GPM3,  GPim1, GPim2,  GVrt2} and for stationary problems \cite{GM, GMit,MR3113415,  GPat,GPatVrt, PV15} (also see the recent surveys \cite{cardaliaguet, MR3195844, bensoussan}). Yet, in the absence of explicit solutions, 
the efficient simulation of  MFG is  importance to many applications. Consequently, researchers have studied numerical algorithms
in various cases, including 
continuous state problems
\cite{AP, achdou2013finite, CDY, DY, MR2928376, MR2928382, MR2974160,
MR3148086, MR2928379, MR2976439, MR3146865} and
 finite state problems \cite{GVW-dual, Gomes:2014kq}. 
Here, we develop numerical methods for (continuous state) stationary MFG using variational and monotonicity methods. 

Certain MFG, called variational MFG, are Euler-Lagrange equations of integral functionals. These MFG are  instances of a wider class of problems -- monotonic MFG. 
In the context of numerical methods, the variational structure of MFG was explored in \cite{CDY}. Moreover, 
monotonicity properties are 
critical for the convergence of the methods in \cite{achdou2013finite, DY, MR2928376}.
Recently, variational and monotonicity methods were used to prove the existence of weak solutions to MFG in, respectively, \cite{FrS, San12, MR3358627, cgbt} and \cite{FG2}. 

Here, our  main contributions are 
two computational approaches for MFG.
For variational MFG, we build 
an approximating method using a gradient flow approach.
This technique gives a simple and efficient algorithm.
Nevertheless, the class of variational MFG is somewhat restricted. 
Monotonic MFG encompass a wider range of problems
that include variational MFG as a particular case. 
In these games, 
the MFG equations involve a monotone nonlinear operator. We use the monotonicity to 
 build a flow that is a contraction in $L^2$ and whose fixed points solve the MFG.

To keep the presentation elementary,  
we develop our methods for the one-dimensional MFG:
\begin{eqnarray}\label{1dsmfg}
\begin{aligned}
\begin{cases}
\displaystyle \frac{u_x^2}{2} + V(x)+b(x) u_x =\ln m+\Hh,\\
-(m (u_x+b(x)))_x =0.
\end{cases}
\end{aligned}
\end{eqnarray}
To streamline the discussion, we study \eqref{1dsmfg} 
with periodic boundary conditions. Thus, the variable $x$
takes values in the one-dimensional torus,
$\Tt$. The potential, $V$, and the drift, $b$, are
given real-valued periodic functions. The unknowns
are $u$, $m$, and  $\Hh$, where \(u\) and \(m
\) are  real-valued periodic functions satisfying $m> 0$, and where   $\Hh$
is a constant. The role of \(\overline H\)
is to allow for $m$ to satisfy $\int_{\Tt} m\,\dx=1$.
Furthermore, 
since adding  an arbitrary constant to $u$ does not change \eqref{1dsmfg}, we require
\begin{equation}
\label{int0}
\int_{\Tt} u \,\dx=0. 
\end{equation}
The system \eqref{1dsmfg} is one of the simplest MFG models. However, its structure is quite rich 
and illustrates our techniques well.
Our methods extend in a straightforward way to other models, including higher-dimensional problems. 
In particular, in Section \ref{impsec}, we discuss applications to
a one-dimensional congestion model and to a two-dimensional MFG.

We end this introduction with a brief outline of our work. In Section~\ref{eprop}, we examine 
various properties of \eqref{1dsmfg}. These properties motivate the 
ideas used in Section~\ref{discrete} to build numerical methods. 
Next, in Section~\ref{impsec}, we discuss the implementation of our approaches and present their numerical simulations.
We conclude this work in Section~\ref{conclusions}
with some final remarks.

\section{Elementary properties}
\label{eprop}

We begin this section by constructing explicit solutions to \eqref{1dsmfg}. 
These are of particular importance for the validation
and comparison of the numerical methods presented
in Section~\ref{discrete}. Next, 
we discuss the variational structure  of \eqref{1dsmfg} and show that \eqref{1dsmfg} is equivalent to the Euler-Lagrange equation of a suitable functional.
Because of this, we introduce a gradient flow approximation and examine some of its elementary properties. Finally, we explain how
 \eqref{1dsmfg} can be seen as a monotone operator.
 This operator induces 
 a flow that is a contraction in $L^2$ and whose stationary points are solutions to \eqref{1dsmfg}. 
 
\subsection{Explicit solutions}

Here, we build explicit solutions to \eqref{1dsmfg}. 
For simplicity, we
assume that $V$ and $b$ are $C^\infty$ functions. Moreover, we identify \(\Tt\) with the interval
\([0,1]\).

Due to the one-dimensional nature of 
\eqref{1dsmfg}, 
if $\int_\Tt b\,\dx=0$,
we have the following explicit solution\begin{align*}
&u(x)=-\int_0^x b(y)\,\dy + \int_\Tt \int_0^x b(y)\,\dy\, \dx,\qquad m(x)=\frac{e^{V(x)-\frac{b^2(x)}{2}}}{\int_{\Tt} e^{V(y)-\frac{b^2(y)}{2}} \,\dy}, \\
& \overline
H=\ln\left( \int_{\Tt}
e^{V(y)-\frac{b^2(y)}{2}} \,\dy\right).
\end{align*}
Suppose that $b=\psi_x$ for some $C^\infty$
 and periodic function $\psi:\Tt\to \Rr$  with $\int_{\Tt} \psi\,\dx=0$.
For
\[
u(x)=\psi(x),\qquad m(x)=\frac{e^{V(x)-\frac{\psi_x^2(x)}{2}}}
{\int_{\Tt} e^{V(y)-\frac{\psi_y^2(y)}{2}} \,\dy},
\qquad H=\ln\left( \int_{\Tt}
e^{V(y)-\frac{\psi_y^2(y)}{2}} \,\dy \right),
\]
the triplet $(u,m, \overline H)$
solves \eqref{1dsmfg}. If $\int_\Tt b\,\dx \neq 0$, we are not aware of any closed-form solution.  
        
Next, we consider the congestion model 
\begin{eqnarray}\label{1dsmfgcon}
\begin{aligned}
\begin{cases}
\displaystyle \frac{u_x^2}{2 m^{1/2}} + V(x) =\ln m+\Hh,\\
-(m^{1/2} u_x)_x =0.
\end{cases}
\end{aligned}
\end{eqnarray}
Remarkably, 
the previous equation has the same solutions as \eqref{1dsmfg} with $b=0$. 
Namely,  
for $u(x)=0$, $m(x)=\frac{e^{V(x)}}{\int_{\Tt} e^{V(y)} \,\dy}$, and \(\overline H = \ln\left(
\int_{\Tt}
e^{V(y)} \,\dy \right) \), the triplet 
 $(u,m, \overline H)$ 
solves \eqref{1dsmfgcon}.

\subsection{Euler-Lagrange equation}

We begin by showing  that \eqref{1dsmfg} is equivalent to the Euler-Lagrange equation of 
the integral functional
\begin{equation}
\label{J}
J[u]= \int_\Tt e^{\frac{u_x^2}{2} + V(x)+b(x) u_x}\, \dx
\end{equation}
defined for $u\in D(J)=   W^{2,2}(\Tt)  \cap   L^2_0(\Tt)$, where $ L^2_0(\Tt)= \{u\in L^2(\Tt)\!: u \hbox{ satisfies \eqref{int0}}\}$.

\begin{remark}[On the domain of $J$]\label{regmin}
As proved in \cite{E1} (see \cite{E2}, \cite{GIMY},  \cite{GPatVrt},  \cite{GPM1}, \cite{GM}, and \cite{PV15} for related problems),
\eqref{1dsmfg} admits a $C^\infty$ solution. By a simple convexity argument, this solution is the unique minimizer of 
{\setlength\arraycolsep{0.5pt}
        \begin{eqnarray*}        \begin{aligned}
        J[u] = \min_{v\in W^{1,2}(\Tt) , \int_\Tt v\,\dx
                =0} J[v].
        \end{aligned}
        \end{eqnarray*}}Thus,  the minimizers of $J$ in $\{ v\in W^{1,2}(\Tt) \!:\, \int_\Tt v\,\dx = 0\}$ are also minimizers of $J$ in $\{ v\in  W^{2,2}(\Tt)
\!:\, \int_\Tt v\,\dx = 0\}$. Thus, the domain of $J$ is not
too restrictive, and, due to this choice, our arguments are substantially simplified. 
In particular, because we are in the one-dimensional setting, $W^{2,2}(\Tt) \subset W^{1,\infty}(\Tt)$. 

We also observe that $L^2_0(\Tt)$ endowed with the $L^2(\Tt)$-inner product is a Hilbert space.
\end{remark}

\begin{lemma}\label{ELcont}
For $\Hh=0$, \eqref{1dsmfg} is equivalent to 
the Euler-Lagrange equation of $J$.
\end{lemma}
\begin{proof} 
Let $u\in  W^{2,2}(\Tt)\cap L^2_0(\Tt)$. 
We say that $u$ is a critical point of $J$ if 
\[
\frac{\d}{\d\epsi} J[u + \epsi v]_{\big|_{\epsi=0}} = 0
\]
for all $v\in W^{2,2}(\Tt) \cap L^2_0(\Tt)$.
Fix any such $v$.
For all $\epsi\in\Rr$, we have that
\begin{eqnarray*}
\begin{aligned}
\frac{\d}{\d\epsi} J[u + \epsi v] = \int_\Tt e^{  \frac{u_x^2}{2} + \epsi u_x v_x + \epsi^2 \frac{v_x^2}{2} + V(x)
+b(x) u_x+\epsi b(x) v_x
 } (u_x v_x + b(x) v_x+\epsi v_x^2) \,\dx.
\end{aligned}
\end{eqnarray*}
Define $m$ by  
\begin{equation}
\label{defm}
\ln m= \frac{u_x^2}{2} + V(x)+b(x)u_x. 
\end{equation}
Then, it follows that
\begin{eqnarray*}
\begin{aligned}
\frac{\d}{\d\epsi} J[u + \epsi v]_{\big|_{\epsi=0}} = 0 \Leftrightarrow  \int_\Tt m(u_x+b(x)) v_x  \,\dx
= 0 \Leftrightarrow  -\int_\Tt (m(u_x+b(x)) )_x v \,\dx = 0.
\end{aligned}
\end{eqnarray*}
Since $v\in W^{2,2}(\Tt)$
is an arbitrary function with zero mean, we conclude
that $u$ is a critical point of $J$ if, and
only if, $(m,u)$ satisfies \eqref{1dsmfg}. 
\end{proof}

As mentioned in Remark~\ref{regmin}, the functional $J$ defined by \eqref{J} 
 admits a unique minimizer. Moreover, since $J$ is convex, any solution to the associated Euler-Lagrange
equation is a minimizer. 
By \eqref{defm}, we have $m> 0$. 
In MFG, it is usual to require 
\[
\int_{\Tt} m \,\dx=1. 
\]
To normalize $m$, we multiplying $m$ by a suitable constant and introduce the parameter $\Hh$, which
leads us to \eqref{1dsmfg}.

\subsection{Monotonicity conditions}

Let $H$ be a Hilbert space with the inner product $\langle\cdot, \cdot\rangle_H$. A map $A:D(A)\subset H\to H$ is a 
monotone operator if
\[
\langle A(w)-A(\tilde w), w-\tilde w\rangle_H \geq 0, 
\]
for all $w, \tilde w\in D(A)$. 

In the Hilbert space $L^2(\Tt)\times L^2(\Tt)$, we define
\begin{equation}
\label{opA}
A
\begin{bmatrix}
m\\
u
\end{bmatrix}
=
\begin{bmatrix}
-\frac{u_x^2}{2} - V(x) -b(x) u_x+\ln m\\
-(m( u_x+ b(x)))_x
\end{bmatrix}
,
\end{equation}
with \(D(A) = \{ (m,u) \in W^{1,2}(\Tt) \times
W^{2,2}(\Tt)\!: \, \inf_\Tt m >0\}\). Observe
that \(A\) maps \(D(A)\) into \(L^2(\Tt)\times L^2(\Tt)\) because
\(W^{1,2}(\Tt)\) is continuously embedded in \(L^\infty(\Tt)\).

\begin{lemma}
        The operator $A$ given by \eqref{opA} is a monotone operator in  $L^2(\Tt)\times L^2(\Tt)$. 
\end{lemma}
\begin{proof}
        Let $(m,u)$, $(\theta, v)\in D(A)\subset  L^2(\Tt)\times L^2(\Tt)$. We have 
        {\setlength\arraycolsep{0.5pt}                \begin{eqnarray*}
                        \begin{aligned}
                                &\left\langle A \begin{bmatrix}m \\                                        u \\
 \end{bmatrix} - A \begin{bmatrix}\theta \\
                                v \\
                        \end{bmatrix}
                        , \begin{bmatrix}m \\
                                u \\
                        \end{bmatrix}
                        - \begin{bmatrix}\theta \\
                                v \\
                        \end{bmatrix} \right\rangle_{L^2(\Tt) \times L^2(\Tt)} \\
                        &\qquad= \int_\Tt (\ln m - \ln \theta)(m-\theta) \,\dx + \int_\Tt\Big( \frac{m}{2} + \frac{\theta}{2}\Big) (u_x - v_x)^2\,\dx,
                \end{aligned}
        \end{eqnarray*}}                where we used integration by parts. Because $\ln(\cdot)$ is an increasing function, and because $\theta, m>0$, the conclusion follows.
\end{proof}

As observed in \cite{ll1,ll3}, 
the monotonicity of  $A$ implies the uniqueness of the solutions.
Here, we use the monotonicity to construct a flow that approximates solutions of \eqref{1dsmfg}. 

\subsection{Weak solutions}
\label{wsd}

Denote by \(\langle \cdot, \cdot \rangle_{\Dd
\times
\Dd'}\) the duality pairing in the sense
of distributions.
 We say that a triplet $(m, u, \Hh)\in \Dd'\times \Dd' \times \RR$ is a weak solution of \eqref{1dsmfg} if 
\begin{equation*}
\left\langle A \begin{bmatrix}
\theta \\
v \\
\end{bmatrix}
-
\begin{bmatrix}\,\overline H\,\\
0 \\
\end{bmatrix}
, 
\begin{bmatrix}\theta \\
v \\
\end{bmatrix}
-
\begin{bmatrix}m \\
u \\
\end{bmatrix}
\right\rangle_{\Dd\times\Dd'}\geq 0
\end{equation*}
for all $(\theta, v)\in \Dd\times \Dd$ satisfying 
$\inf_\Tt
\theta>0$ and $\int_{\Tt} \theta\,\dx=1$. 
\subsection{Continuous gradient flow}
Next, we introduce the gradient flow
of the energy  \eqref{J}
with respect to the $L^2(\Tt)$-inner product. First, 
we extend $J$ in \eqref{J} to the whole space $L^2_0(\Tt)$ by setting $J[u]=+\infty$ if $u \in L^2_0(\Tt) \backslash W^{2,2}(\Tt)$. We will not relabel this extension. 

The functional $J: L^2_0(\Tt) \to [0,+\infty]$ is  proper, convex, and lower semicontinuous in $L^2_0(\Tt)$.
The subdifferential of $J$ is the map $\partial J: L^2_0(\Tt)\to 2^{L^2_0(\Tt)}$ defined for $u\in L^2_0(\Tt)$ by
{\setlength\arraycolsep{0.5pt}\begin{eqnarray*}
\begin{aligned}
\partial J[u]= \big\{ v\in L^2_0(\Tt)\!:\, J[w] \geq J[u] + \langle v, w-u \rangle_{L^2(\Tt)} \hbox{ for all } w\in L^2_0(\Tt)\big\}. 
\end{aligned}
\end{eqnarray*}}
The domain of $\partial J$, $D(\partial J)$, is the set of all
$u\in  L^2_0(\Tt)$ such that $\partial J[u] \not= \emptyset$.

The gradient flow with respect to the $L^2(\Tt)$-inner product and 
energy  $J$ is
{\setlength\arraycolsep{0.5pt}\begin{eqnarray}\label{gfabst}
\begin{aligned}
\dot{\boldsymbol{ u}}(t) \in -\partial J [\boldsymbol{u}(t)],\quad t\geq0,
\end{aligned}
\end{eqnarray}}where $\boldsymbol{u}: [0,+\infty) \to L^2_0(\Tt)$. As we will see next, \eqref{gfabst} is equivalent to
{\setlength\arraycolsep{0.5pt}\begin{eqnarray}\label{ourgf}
\begin{aligned}
\dot{\boldsymbol{ u}}(t)= \big(\boldsymbol{m}(t)((\boldsymbol{u}(t))_x+ b(x))\big)_x,\quad
t\geq0,
\end{aligned}
\end{eqnarray}}where $\boldsymbol{m}(t)$ is given by \eqref{defm} with $u$ replaced by $\boldsymbol{u}(t)$. Moreover, if the solution $\boldsymbol{u}$ to \eqref{ourgf} is regular enough, then 
\[
\frac{\d}{\dt}J[\boldsymbol{u}]=-\int_{\Tt} \Big[\big(\boldsymbol{m} (\boldsymbol{u}_x+b(x))\big)_x\Big]^2
\,\dx\leq 0. 
\]

\begin{proposition}\label{onsubdif}
We have $D(\partial J) =   W^{2,2}(\Tt)
 \cap   L^2_0(\Tt)$ and, for $u\in D(\partial J)$, $\partial J[u] = -\big (m (u_x+b(x))\big)_x$, where $m$ is given by \eqref{defm}.   
\end{proposition}

\begin{proof} By \cite[Thm.~1 in \S9.6.1]{E6}, $D(\partial J) \subset D(J)=\{u\in L^2_0(\Tt)\!:\, J[u]<+\infty\}=   W^{2,2}(\Tt)
 \cap   L^2_0(\Tt)$. 

Conversely, fix $u\in W^{2,2}(\Tt)
 \cap   L^2_0(\Tt)$,  let  $m$ be given by \eqref{defm},
and set \(v= -\big(m (u_x+b(x))\big)_x\). Then,
\(v\in L^2_0(\Tt)\)
by the embedding $W^{2,2}(\Tt) \subset
W^{1,\infty}(\Tt)$ and by the periodicity  of \(u\), \(m\), and \(b\).
Moreover, using  the convexity of the exponential function, the integration by parts formula, and
the conditions $m>0$ and $\frac{w_x^2}{2}
- \frac{u_x^2}{2} \geq u_xw_x - u_x^2$, for each
$w\in W^{2,2}(\Tt)
 \cap   L^2_0(\Tt)$,  we obtain
{\setlength\arraycolsep{0.5pt}\begin{eqnarray*}
\begin{aligned}
J[w] &\geq J[u] + \int_\Tt m \Big( \frac{w_x^2}{2} + b(x) w_x -  \frac{u_x^2}{2}
- b(x) u_x \Big)\, \dx\\
& \geq J[u] + \int_\Tt m \big(u_x w_x - u_x^2  + b(x)(w_x - u_x)\big)\,\dx \\
 & = J[u] - \int_\Tt \big(m (u_x+b(x))\big)_x (w-u)\,\dx.
\end{aligned}
\end{eqnarray*}} Because $w\in W^{2,2}(\Tt)
 \cap   L^2_0(\Tt)$ is arbitrary, and because $J=+\infty$ in $L^2_0(\Tt) \backslash
W^{2,2}(\Tt)$, we obtain $v=-(m (u_x+b(x)))_x
\in \partial J[u]$. Because  $u\in W^{2,2}(\Tt)
 \cap   L^2_0(\Tt)$ is also arbitrary, we get $D(\partial J) \supset W^{2,2}(\Tt)
 \cap   L^2_0(\Tt)$.
 
 To conclude the proof, we show that for $u\in D(\partial
J)$, the function $-\big (m (u_x+b(x))\big)_x$ with $m$  given by \eqref{defm} is the unique element of  $ \partial J[u] $. Let $\bar v\in \partial J[u] $.  Then, for all $\epsi>0$ and $w\in W^{2,2}(\Tt)
 \cap   L^2_0(\Tt)$, we have 
\[
\frac{J[u \pm\epsi w] - J[u]}{\epsi} \geq \pm \langle\bar  v, w \rangle_{L^2(\Tt)}.
\]
Letting $\epsi\to 0^+$ and arguing as in the proof of Lemma~\ref{ELcont}, we obtain
\[
\big\langle-\big(m (u_x+b(x))\big)_x, w\big \rangle_{L^2(\Tt)}
\geq \pm
\langle\bar  v, w\rangle_{L^2(\Tt)}.
\]
Because $w\in W^{2,2}(\Tt)
 \cap   L^2_0(\Tt)$ is arbitrary, it follows that $\bar v= -\big(m (u_x+b(x))\big)_x$. 
\end{proof}

The following result about solutions to the gradient
flow \eqref{ourgf} holds by  \cite[Thm.~3 in \S9.6.2]{E6}
and by the fact that  $W^{2,2}(\Tt)
 \cap   L^2_0(\Tt)$ is dense in $L^2_0(\Tt)$.

\begin{theorem}For each $u\in W^{2,2}(\Tt)
 \cap   L^2_0(\Tt)$, there exists a unique function $\boldsymbol{u} \in C([0,+\infty);L^2_0(\Tt))$, with $\dot{\boldsymbol{ u}} \in L^\infty(0,+\infty;L^2_0(\Tt))$, such that
\begin{itemize}
\item[(i)] $\boldsymbol{u}(0) =u,$

\item [(ii)] $\boldsymbol{u}(t) \in W^{2,2}(\Tt)
 \cap   L^2_0(\Tt)$ for each $t>0$,
 
 \item [(iii)] $\dot{\boldsymbol{ u}}(t) = \big(\boldsymbol{m}(t)
 ((\boldsymbol{u}(t))_x+
b(x))\big)_x$ for a.e.\! $t\geq0$, where $\boldsymbol{m}(t)$ is given by \eqref{defm}
with $u$ replaced by $\boldsymbol{u}(t)$.
\end{itemize}
\end{theorem}

\subsection{Monotonic flow}

Because the operator $A$ is monotone,
the flow
\begin{equation}
\label{dynapprox}
\begin{bmatrix}
\dot{\boldsymbol{ m}}\\
\dot{\boldsymbol{ u}}
\end{bmatrix}
=
-A
\begin{bmatrix}
\boldsymbol{ m}\\
\boldsymbol{ u}
\end{bmatrix}
\end{equation}
is a contraction in $L^2(\Tt) \times L^2(\Tt)$. That is, if $(\boldsymbol m,\boldsymbol u)$ and $(\tilde {\boldsymbol m}, \tilde {\boldsymbol u})$ solve \eqref{dynapprox}, then
\[
\frac \d {\dt} \left(\|\mbfu-\tilde \mbfu\|^2_{L^2(\Tt)}+\|\mbfm-\tilde \mbfm\|^2_{L^2(\Tt)}\right)
=-2
\left\langle
 A \begin{bmatrix} \mbfm \\
                   \mbfu \\
   \end{bmatrix} - 
A \begin{bmatrix}\tilde \mbfm \\
                                \tilde \mbfu \\
                        \end{bmatrix}
                        , \begin{bmatrix} \mbfm \\
                                \mbfu \\
                        \end{bmatrix}
                        - \begin{bmatrix}\tilde \mbfm \\
                                \tilde \mbfu \\
                        \end{bmatrix}
\right\rangle_{L^2(\Tt) \times L^2(\Tt)}\leq 0,
\]
provided that for each \(t \geq 0\), \((\mbfm(t),\mbfu(t))\),
 \((\tilde \mbfm(t), \tilde \mbfu(t)) \in D(A)\). The flow \eqref{dynapprox} has two undesirable features. First, it does not preserve probabilities; second, the flow may not preserve the condition $\mbfm>0$. To conserve probability, we modify \eqref{dynapprox} through
\begin{equation}
\label{dynapprox2}
\begin{bmatrix}
\dot \mbfm\\
\dot \mbfu
\end{bmatrix}
=
-A
\begin{bmatrix}
\boldsymbol m\\
\boldsymbol u
\end{bmatrix}
+
\begin{bmatrix}
\,\Hh(t)\\
0
\end{bmatrix}
,
\end{equation}
where $\Hh(t)$ is such that $\frac \d {\dt} \int_\Tt \boldsymbol m\,\dx=0$.
A straightforward computation shows that \eqref{dynapprox2} is still a contraction in $L^2(\Tt) \times L^2(\Tt)$. More precisely,
 if $(\boldsymbol
m,\boldsymbol u)$ and $(\tilde {\boldsymbol m},
\tilde {\boldsymbol u})$ solve \eqref{dynapprox2}
and satisfy \((\mbfm(t),\mbfu(t))\),
 \((\tilde \mbfm(t), \tilde \mbfu(t)) \in D(A)\)
  for all \(t\geq0\), \(\int_\Tt \boldsymbol m(0)\,\dx =1\),
and \(\int_\Tt \tilde{\boldsymbol m}(0)\,\dx =1\),
then
\[
\frac \d {\dt} \left(\|\mbfu-\tilde \mbfu\|^2_{L^2(\Tt)}+\|\mbfm-\tilde
\mbfm\|^2_{L^2(\Tt)}\right)
\leq 0.
\]
Furthermore, positivity holds for the discretization of  \eqref{dynapprox2}
that we develop in the next section. Therefore, the discrete analog of  \eqref{dynapprox2} is a contracting flow that preserves probability and positivity. Then, as $t\to \infty$, the solutions approximate \eqref{1dsmfg}.

\section{Discrete setting}\label{discrete}

Here, we discuss the numerical approximation of \eqref{1dsmfg}. We use a monotone scheme for the Hamilton-Jacobi
equation. For the Fokker-Planck equation, we consider
the adjoint of the linearization of the
discrete Hamilton-Jacobi equation. This technique preserves
both the gradient structure and the monotonicity properties of
the original problem.

\subsection{Discretization of the Hamilton-Jacobi operator}
We consider $N$ equidistributed points on $[0,1]$, 
$x_i=\frac{i}{N}$, $i\in\{1,...,N\}$,
and corresponding values of the approximation
to $u$
given by the vector
 $u=(u_1, ..., u_N)\in\Rr^N$.
We set $h=\frac{1}{N}$. 
To incorporate the periodic conditions,
 we use the periodicity convention $u_0=u_N$ and $u_{N+1}= u_1$.
For each $i\in \{1,...,N\}$, let $\psi_i: \RR^N\to\RR^2$ be given by
\begin{eqnarray*}
\begin{aligned}
\psi_i (u) = (\psi_i^1(u),\psi_i^2(u)) = \Big(\frac{u_i - u_{i+1}}{h}, \frac{u_i - u_{i-1}}{h}
\Big)
\end{aligned}
\end{eqnarray*}
for $u\in\RR^N$. To discretize the operator
\[
u\mapsto \frac{u_x^2}{2} + V(x)+b(x) u_x, 
\]
we use a monotone finite difference scheme, see \cite{BS}.
This scheme is built as follows.
 We consider a function $F:\Rr\times \Rr\times\Tt\to \Rr$ satisfying
the following four conditions.\begin{enumerate}
  \item $F(p,q,x)$ is jointly convex in $(p,q)$.
  \item The functions $p\mapsto F(p, q, x)$ for fixed $(q, x)$ and  $q\mapsto F(p, q, x)$ for fixed $(p, x)$
  are increasing. 
  \item $F(-p, p, x)=\frac{p^2}{2}+b(x)p+V(x)$.
  \item There exists a positive constant, \(c\),
  such that 
  \begin{equation}
  \label{newonF}
  F(-p, q, x) + F(q', p, x') \geq
  - \frac 1 c + c\,p^2.
  \end{equation}  
\end{enumerate}
An example of such a function  may be found
in Section~\ref{impsec} below.
Next, we set
\begin{equation}
\label{Fi}
F_i(p,q)=F(p,q, x_i). 
\end{equation}
Let $G:\RR^N\to\RR^N$ be the function defined
for $u\in\RR^N$
by 
\begin{equation}
\label{GFormula}
G(u) = (G_1(u), ... , G_N(u))= \big((F_1\circ
\psi_1)(u), ... , (F_N\circ
\psi_N)(u)\big).
\end{equation}
Then, $G(u)$ is a finite difference scheme for the Hamilton-Jacobi operator 
$\frac{u_x^2}{2}+V(x)+b(x)u_x$. 

\begin{remark}
In the higher-dimensional case, the Hamilton-Jacobi operator can be discretized 
with a similar monotone scheme. See \cite{AO} for a systematic study of convergent monotone difference
schemes for elliptic and parabolic equations. 
\end{remark}

\subsection{The variational formulation}

Here, we study the following discrete version,
\(\phi:\RR^N \to \RR\), of the functional \eqref{J}:
\begin{equation}
\label{JD}
\phi(u)= \sum_{i=1}^{N} h\, e^{G_i(u)},\quad u\in\RR^N,
\end{equation}
where $G_i$ is given by \eqref{GFormula}.

\begin{lemma}\label{phicx}
The function $\phi$ given by \eqref{JD} 
is convex.
\end{lemma}
 \begin{proof}
Fix $\lambda\in (0,1)$ and $u$, $v\in \Rr^N$.
Because each $F_i$ is convex,
because the exponential is an increasing
convex function, and because $h>0$, we have
\begin{eqnarray*}
\begin{aligned}
\phi(\lambda u + (1-\lambda)v) 
&= \sum_{i=1}^{N} h\, e^{F_i
\big( \lambda \frac{u_i - u_{i+1}}{h} + (1-\lambda)
\frac{v_i - v_{i+1}}{h}, \lambda \frac{u_i - u_{i-1}}{h}  + (1-\lambda) \frac{v_i
- v_{i-1}}{h}
\big) } \\
& \leq \sum_{i=1}^{N} h \, e^{\lambda
F_i
\big(\frac{u_i - u_{i+1}}{h}, \frac{u_i - u_{i-1}}{h}
\big) + (1-\lambda) F_i
\big(\frac{v_i - v_{i+1}}{h}, \frac{v_i - v_{i-1}}{h}
\big)} \\
& \leq \lambda \phi(u) + (1-\lambda) \phi(v),
\end{aligned}
\end{eqnarray*}
which completes the proof.\end{proof}

\begin{lemma}
        \label{L33}
Let $\phi$ be given by \eqref{JD}. 
Let $\Ll^*_u:\RR^N\to\RR^N$
represent the adjoint operator of the linearized
operator
$\Ll_u :\RR^N\to\RR^N$ of the function $G$
at $u\in\RR^N$.
A vector $u\in \RR^N$ is a critical point of $\phi$ if and only if there exists $\tilde m\in\RR_+^N$ such that the pair  $(\tilde m,u)$ satisfies 
\begin{equation*}
\begin{cases}
\displaystyle G_i(u)= \ln \tilde m_i, \\
\displaystyle (\Ll_u^* \tilde m)_i=0
\end{cases}
\end{equation*}
for all $i\in \{1,...,N\}$.
\end{lemma}

\begin{proof}
The proof is similar to the one for Lemma~\ref{ELcont}.
\end{proof}

\begin{remark}We observe that for $w\in \RR^N$ and $i\in\{1,...,N\}$, we have
\begin{eqnarray*}
\begin{aligned}
(\Ll_u^* w)_i &= \sum_{j=1}^N \partial_iG_j(u)w_j
 = \sum_{j=i-1}^{i+1} \frac{\partial}{\partial
u_i} F_j(\psi_j(u))w_j \\
& = \frac{1}{h} \bigg[ - \frac{\partial  F_{i-1}}{\partial p}
(\psi_{i-1}(u))w_{i-1} + \frac{\partial F_{i}}{\partial
p} (\psi_{i}(u))w_{i} + \frac{\partial F_{i}}{\partial
q} (\psi_{i}(u))w_{i} - \frac{\partial F_{i+1}}{\partial
q} (\psi_{i+1}(u))w_{i+1}\bigg].
\end{aligned}
\end{eqnarray*}
Simple computations show that $\Ll_u^* w$ is a consistent finite difference scheme for the Fokker-Planck equation. 
\end{remark}

\subsection{The discretized operator}

Motivated by the previous discussion,
we discretize \eqref{1dsmfg} through the finite difference operator
        \begin{equation}
        \label{A}
                \begin{aligned}
                        A^N \begin{bmatrix}m \\
                                u \\
                        \end{bmatrix}
                        = \begin{bmatrix}\displaystyle -G(u) + \ln m  \\
                                \Ll_u^* m
                        \end{bmatrix}, \quad (m,u)\in
\RR^N_+ \times \RR^N,                 \end{aligned}
        \end{equation}
where $\ln m = (\ln m_1, ... , \ln m_N)$ and
where $G$ is given by \eqref{GFormula}. 
Accordingly, the analog to \eqref{1dsmfg} becomes 
\begin{equation}
\label{A=0}
A^N\begin{bmatrix}m^N \\
u^N \\
\end{bmatrix}
=
\begin{bmatrix}-\Hh^N \boldsymbol{\iota} \\
0 \\
\end{bmatrix},
\end{equation}
where we highlighted the dependence on \(N\)
and where \(\boldsymbol{\iota}=(1, ..., 1)\in
\RR^N\).
In \eqref{A=0}, the unknowns are the vector $u^N$, 
the discrete probability density $m^N$, normalized to
 $h \sum_{i=1}^N m_i^N=1$, and the effective Hamiltonian $\Hh^N$. 

We are interested in two main points. The first
is the existence and approximation of solutions to
\eqref{A=0}. 
The second is the convergence of these solutions to solutions of \eqref{1dsmfg}.
The first issue will be examined by gradient-flow techniques and by monotonicity methods. The second issue 
is a consequence of a modified Minty method. 

\subsection{Existence of solutions}

Here, we prove the existence of solutions to \eqref{A=0}. Our proof uses ideas similar to those of the direct method of the calculus of variations.
\begin{proposition}
        \label{eos}
Let $\phi$ be as in \eqref{JD}. Then, there exists $u^N\in \Rr^N$ with $\sum_{i=1}^N u_i^N=0$ that minimizes $\phi$. 
Moreover, 
\begin{equation}
\label{l2b}
h\sum_{i=1}^{N} (u_i^N)^2\leq C
\end{equation}
for some positive constant \(C\) independent of $h$. In addition, there exist \(m^N\in \RR^N\)
with \(h\sum_{i=1}^N m_i^N = 1\) and \(\overline
H^N \in\RR^N\) such that the triplet \((u^N, m^N,
\overline H^N)\) satisfies \eqref{A=0}.
  \end{proposition}
\begin{proof}
To simplify the notation, we will drop the explicit
dependence on \(N\) of \(u^N\) and \(m^N\). Accordingly,
we simply write \(u\) and \(m\).

 As in the direct method of the calculus of variations, we select a minimizing sequence, $(u^k)_{k\in\NN}
\subset \Rr^N$, for $\phi$ satisfying
\begin{equation}
\label{s=0}
\sum_{i=1}^N u_i^k=0.
\end{equation}
Then, there exists a positive constant, $C$, independent of $k$ and \(h\) such that $\sup_{k\in\NN}\phi(u^k)\leq C$. Using Jensen's inequality, for all \(k\in\NN\),
we have
that\[
h\sum_{i=1}^{N} G_i(u^k) \leq \tilde C,
\]  
where $\tilde C$ is  positive constant that
is independent of $k$ and \(h\). This  estimate together
with \eqref{newonF}--\eqref{GFormula}  implies
that\[
\sum_{i=1}^{N} \frac{|u_{i+1}^k-u_i^k|^2}{h}\leq \bar C
\]
for some positive constant $\bar C$ that is
independent
of $k$ and \(h\).
By a telescoping series argument combined with
the Cauchy inequality, for all \(l,m\in \{1,...,N\}\),
we have
\begin{align*}
|u_l^k-u_m^k|&\leq \sum_{i=1}^{N} |u_{i-1}^k-u_i^k|\leq
\Big(\frac{1}{h}\Big)^{\frac 1 2} \left(\sum_{i=1}^{N} |u_{i-1}^k-u_i^k|^2\right)^{\frac 1 2}\leq
\bar C^{\frac 1 2}. 
\end{align*}
The previous bound combined with \eqref{s=0} yields
\begin{equation*}
\max_{1 \leq i \leq N} |u_i^k| \leq \bar C^{\frac 1 2}.
\end{equation*}
 By compactness and by extracting a subsequence if necessary, there exists
$u\in \Rr^N$ with $\sum_{i=1}^N u_i=0$ such that $u^k\to u$ in \(\RR^N\). The continuity of $\phi$ implies that $u$ is a minimizer of \(\phi\). Furthermore, \eqref{l2b} holds.

Finally, by Lemma~\ref{L33}, we have 
\[
A^N\begin{bmatrix}
\tilde m \\
u \\
\end{bmatrix}
=
\begin{bmatrix} 0\\
0 \\
\end{bmatrix}
\]
for $\tilde m_i=e^{G_i(u)}$. By selecting $\Hh^N$ conveniently 
and by setting $m_i=e^{-\Hh^N}\tilde m_i$, we obtain $h\sum_{i=1}^{N} m_i=1$. Moreover, the
triplet \((u,m, \overline H^N) \) satisfies \eqref{A=0}.
\end{proof}

\subsection{Monotonicity properties}

Next, we prove that the operator $A^N$ is monotone.

\begin{lemma}        The operator $A^N$ given by \eqref{A} is monotone in \(\RR^N \times \RR^N\). More precisely, for all $(m,u)$, $(\theta,v) \in  \RR_+^N \times \RR^N$,
                \begin{eqnarray}\label{ANmon}
                \begin{aligned}
                        \bigg\langle A^N \begin{bmatrix}m \\
                                u \\
                        \end{bmatrix} - A^N \begin{bmatrix}\theta \\
                        v \\
                \end{bmatrix}
                , \begin{bmatrix}m \\
                        u \\
                \end{bmatrix}
                - \begin{bmatrix}\theta \\
                        v \\
                \end{bmatrix} \bigg\rangle_{\RR^N \times \RR^N} \geq 0.
        \end{aligned}
\end{eqnarray}
\end{lemma}

\begin{proof}
        Fix $(m,u)$,
        $(\theta,v) \in  \RR_+^N \times \RR^N$. Using the definition of \(A^N\) and the fact that
 $\ln(\cdot)$ is increasing, we obtain
                         \begin{eqnarray*}
                \begin{aligned}
                        & \bigg\langle A^N \begin{bmatrix}m \\
                                u \\
                        \end{bmatrix} - A^N \begin{bmatrix}\theta \\
                        v \\
                \end{bmatrix}
                , \begin{bmatrix}m \\
                        u \\
                \end{bmatrix}
                - \begin{bmatrix}\theta \\
                        v \\
                \end{bmatrix} \bigg\rangle_{\RR^N \times \RR^N}\\
                                & \qquad= \sum_{i=1}^{N} \Big[\big(G_i(v) - G_i(u) + \ln m_i - \ln \theta_i\big) (m_i - \theta_i)  + \big((\Ll_u^* m)_i - (\Ll_v^*
                \theta)_i\big) (u_i - v_i) \Big]\\
                                &\qquad \geq \sum_{i=1}^{N} \big[  \big(G_i(v)-G_i(u)\big)  m_i   + (\Ll_u^* m)_i  (u_i - v_i)\big] \\
                                &\qquad\qquad + \sum_{i=1}^{N} \big[\big(G_i(u) - G_i(v)\big) \theta_i   - (\Ll_v^* \theta)_i  (u_i - v_i)\big].
        \end{aligned}
\end{eqnarray*}
 Moreover, by the periodicity convention, we have
that\begin{eqnarray*}
        \begin{aligned}
                \sum_{i=1}^{N}(\Ll_u^* m)_i  (u_i - v_i) &= \frac{1}{h} \sum_{i=1}^{N} \bigg[ - \frac{\partial  F_{i-1}}{\partial
                        p}
                (\psi_{i-1}(u))m_{i-1} (u_i - v_i)+ \frac{\partial F_{i}}{\partial
                        p} (\psi_{i}(u))m_{i} (u_i - v_i) \bigg] \\
                                &\qquad +\frac{1}{h} \sum_{i=1}^{N} \bigg[  \frac{\partial F_{i}}{\partial
                        q} (\psi_{i}(u))m_{i} (u_i - v_i)- \frac{\partial F_{i+1}}{\partial
                        q} (\psi_{i+1}(u))m_{i+1} (u_i - v_i)\bigg]\\
                                &=
                \frac{1}{h} \sum_{i=1}^{N} \bigg[ - \frac{\partial  F_{i}}{\partial
                        p}
                (\psi_{i}(u))m_{i} (u_{i+1} - v_{i+1})+ \frac{\partial
                        F_{i}}{\partial
                        p} (\psi_{i}(u))m_{i} (u_i - v_i) \bigg] \\
                                &\qquad +\frac{1}{h} \sum_{i=1}^{N} \bigg[
                \frac{\partial F_{i}}{\partial
                        q} (\psi_{i}(u))m_{i} (u_i - v_i)- \frac{\partial  F_{i}}{\partial
                        q}
                (\psi_{i}(u))m_{i} (u_{i-1} - v_{i-1})\bigg]\\
                                &= \sum_{i=1}^{N} \bigg[ \frac{\partial
                        F_{i}}{\partial
                        p}
                (\psi_{i}(u)) (\psi_i^1(u) - \psi_i^1(v)) + \frac{\partial
                        F_{i}}{\partial
                        q}
                (\psi_{i}(u)) (\psi_i^2(u) - \psi_i^2(v)) \bigg] m_i \\
                                &= \sum_{i=1}^{N} \nabla
                F_{i}
                (\psi_{i}(u)) \cdot (\psi_i(u) - \psi_i(v))
                m_i.
        \end{aligned}
\end{eqnarray*}
So,
the estimate\begin{align}
\notag
                & \sum_{i=1}^{N} \big[\big(G_i(v)-G_i(u)\big)
                m_i   + (\Ll_u^* m)_i  (u_i - v_i)\big] \\\label{mnt}
                                &\qquad =  \sum_{i=1}^{N} \big[F_i(\psi_i(v))-F_i(\psi_i(u))-\nabla
                F_{i}
                (\psi_{i}(u)) \cdot (\psi_i(v) - \psi_i(u))\big] m_i\geq 0
\end{align}
follows from the convexity of each $F_i$ and from the positivity of each $m_i$. Similarly,
\begin{eqnarray*}
        \begin{aligned}
                \sum_{i=1}^{N} \big[\big(G_i(u)
                - G_i(v)\big) \theta_i   - (\Ll_v^* \theta)_i
                (u_i - v_i)\big] \geq 0,
        \end{aligned}
\end{eqnarray*}
which concludes the proof.
\end{proof}

\begin{remark}
\label{3p7}
Because the operator $A^N$ is monotone, $(u^N, m^N, \Hh^N)$ solves 
\eqref{A=0} if and only if the condition 
\begin{equation}
\label{weaksol}
\left\langle A^N \begin{bmatrix}
\theta \\
v \\
\end{bmatrix}
-
\begin{bmatrix}\Hh^N\boldsymbol{\iota}\\\
0 \\
\end{bmatrix}
, 
\begin{bmatrix}\theta \\
v \\
\end{bmatrix}
-
\begin{bmatrix}m^N \\
u^N \\
\end{bmatrix}
 \right\rangle_{\RR^N \times \RR^N}\geq 0 
\end{equation}
holds for every
$(v, \theta)\in \RR^N \times \RR^N_+$.
\end{remark}

\begin{definition}
\label{3p8}
We say that $A^N$ is strictly monotone if \eqref{ANmon} holds with strict inequality 
whenever  $(m,u)$, $(\theta,v) \in  \RR_+^N \times
\RR^N$ satisfy $v\neq u$ and $\sum_{i=1}^{N} v=\sum_{i=1}^{N} u$. 
\end{definition}

\subsection{Uniform estimates}

Estimates that do not depend on $N$ play a major role in establishing the convergence of 
solutions of \eqref{A=0} to \eqref{1dsmfg}. Here, we prove
elementary energy estimates that are sufficient to show convergence. 

\begin{pro}
\label{uest}
Let $(u^N,m^N, \Hh^N)$ solve \eqref{A=0}.
Further assume that $\sum_{i=1}^{N} u^N_i=0$.
Then, there exists a positive constant, $C$,  independent of $N$  such that 
\begin{equation}
\label{ll2b}
\frac 1 N \sum_{i=1}^{N} (u_i^N)^2\leq C
\end{equation}
and
\begin{equation*}
\left|\Hh^N\right|\leq C. 
\end{equation*}
\end{pro}
\begin{proof}
By Lemma~\ref{L33}, we have that \(\phi(u^N) \leq C\), where \(C=
\phi(0)\). Then, arguing as in the proof of Proposition \ref{eos}, we obtain the $\ell^2$ bound in \eqref{ll2b}.

The bound for $\Hh^N$ is proven in two steps. First, we have
\[
\frac 1 N \sum_{i=1}^{N} G_i(u^N)=\Hh^N+\frac 1 N \sum_{i=1}^{N} \ln (m_i^N)\leq \Hh^N 
\]
by Jensen's inequality. Because $G_i$ is bounded from below, we obtain
\[
\Hh^N\geq -C
\]
for some  constant $C\geq0$  independent
of $N$. Second, for each \(i\in \{1,..., N\}\), we multiply the $i$th equation in \eqref{A=0} by $m_i$ and the \((N+i)\)th equation by $-u_i$. Adding the resulting expressions
and summing over $i$, we get
\[
\frac 1 N \sum_{i=1}^{N} m_i\left(G_i(u)-(\Ll_u u)_i\right)=\Hh^N+\frac 1 N \sum_{i=1}^{N} m_i \ln m_i\geq \Hh^N 
\]
by Jensen's inequality. By the concavity of $G_i(u)$, we have
\[
G_i(u)-(\Ll_u u)_i=G_i(u)+(\Ll_u (0-u))_i\leq G_i(0).
\]
Hence, $\Hh^N\leq C$ for some constant \(C>0\)
independent of $N$.         
\end{proof}

\begin{remark}
The proof of the previous proposition gives an $\ell^\infty$ bound for $u_i$, not just the $\ell^2$ bound in \eqref{ll2b}. 
However, the technique used in the proof is one-dimensional since it is similar to the proof of the one-dimensional Morrey's theorem. 
As stated in the proposition,  inequality \eqref{ll2b} is a discrete version of the Poincar\'e inequality;  this inequality holds in any dimension. 
Finally,  
for our purposes, \eqref{ll2b} is sufficient. 
\end{remark}

\subsection{Convergence}

Here, we show the convergence of solutions of \eqref{A=0} to weak solutions of \eqref{1dsmfg}.

\begin{proposition}
For $N\in \Nn$, let $(u^N,m^N, \Hh^N)\in \Rr^N\times \Rr^N\times \RR$ be a solution of \eqref{A=0}. Denote by $\bar u^N$ the step function in $[0,1]$
that takes the value $u^N_i$    in the interval $\left[\frac {i-1} N, \frac i N\right]$ for $1\leq i\leq N$. Similarly, 
$\bar m^N$ is the step function associated with $m^N$. 
Then,  up to a (not relabeled) subsequence,  $\Hh^N\to \Hh$ in $\Rr$, $\bar u^N\rightharpoonup \bar u$ in $L^2([0,1])$, and
$\bar m^N\rightharpoonup \bar m$ in $\mathcal{P}([0,1])$. Moreover, 
$(\bar m, \bar u, \Hh)$ is a weak solution of \eqref{1dsmfg}.   
\end{proposition}
\begin{proof}
According to Proposition \ref{uest}, $|\Hh^N|$ and $\|\bar u^N\|_{L^2([0,1])}$ are uniformly bounded with respect to \(N\). Moreover, $\|\bar m^N\|_{L^1([0,1])}=1$ by construction. Therefore, 
there exist $\Hh\in \Rr$, $\bar u \in L^2([0,1])$, and $\bar m\in \mathcal{P}([0,1])$ such that $\Hh^N\to \Hh$ in $\Rr$, $\bar u^N\rightharpoonup \bar u$ in $L^2([0,1])$, and
$\bar m^N\rightharpoonup \bar m$ in $\mathcal{P}([0,1])$, up to a (not relabeled) subsequence. 

Select $v, \theta \in C^\infty(\Tt)$ satisfying
$\theta> 0$ and $\int_\Tt \theta\,\dx=1$. Set $v^N_i=v\left(\frac i N\right)$ and $\theta^N_i=\theta\left(\frac i N\right)$.
Then, by Remark \ref{3p7}, 
\begin{align*}
0&\leq \left\langle A^N \begin{bmatrix}
        \theta^N \\
        v^N \\
\end{bmatrix}
-
\begin{bmatrix}\Hh^N\\
        0 \\
\end{bmatrix}
, 
\begin{bmatrix}\theta^N \\
        v^N \\
\end{bmatrix}
-
\begin{bmatrix}m^N \\
        u^N \\
\end{bmatrix}
\right\rangle_{\RR^N \times \RR^N}
\\
&=
O\left(\frac 1 N\right)+
\left\langle A \begin{bmatrix}
\theta \\
v \\
\end{bmatrix}
-
\begin{bmatrix}\Hh^N\\
0 \\
\end{bmatrix}
, 
\begin{bmatrix}\theta \\
v \\
\end{bmatrix}
-
\begin{bmatrix}\bar m^N \\
\bar u^N \\
\end{bmatrix}
\right\rangle_{\Dd\times\Dd'}.
\end{align*}
The proposition follows by letting $N\to \infty$ in this last expression.  
        \end{proof}

\subsection{A discrete gradient flow}

To approximate \eqref{A=0}, we consider two approaches.
Here, we discuss 
a gradient-flow approximation. Later, we examine a monotonicity-based
method. 

The discrete-time gradient flow is
\begin{equation}
\label{gflow}
\dot \mbfu=-(\Ll_{\mbfu^*} \tilde \mbfm)_i, 
\end{equation}
where $\tilde \mbfm_i=e^{G_i(\mbfu)}$.
Because $\phi$ is convex, $\phi(\mbfu(t))$ is decreasing. Moreover, the proof of proposition \eqref{eos} shows that $\phi$ is coercive
on the set $\sum_{i=1}^N \mbfu_i=0$. Note that \eqref{gflow} satisfies
\[
\frac \d {\dt} \sum_{i=1}^{N} \mbfu_i=0.
\]
Consequently, $\mbfu(t)$ is bounded and converges to a critical point of $\phi$. Finally, we obtain a solution 
to \eqref{A=0} by normalizing $ \tilde \mbfm_i$. 

\subsection{Dynamic approximation}

        We can use  the monotonicity of  $A^N$  to build a contracting flow in $L^2$ whose fixed points satisfy \eqref{A=0}. This flow
is
\[
\begin{bmatrix}
\dot \mbfm\\
\dot \mbfu
\end{bmatrix}
=
-A^N
\begin{bmatrix}
\mbfm\\
\mbfu
\end{bmatrix}
+
\begin{bmatrix}
\Hh^N(t) \boldsymbol{\iota}\\
0
\end{bmatrix}, 
\]
where $\Hh^N(t)$ is such that the total mass is preserved; that is,
\[
\sum_{i=1}^{N} \dot \mbfm_i=0.
\]
Due to the logarithmic nonlinearity, $\mbfm(t)
> 0$ for all $t$. We further observe that 
\[
\frac \d {\dt} \sum_{i=1}^{N} \mbfu_i=0.
\]
Moreover, if $(\bar m^N, \bar u^N, \Hh^N)$ is a solution of \eqref{A=0}, then the monotonicity of $A^N$ implies that
\[
\frac \d {\dt} \left(\|\mbfm-\bar m^N\|^2+\|\mbfu-\bar u^N\|^2\right)\leq 0. 
\]
Furthermore, if strong monotonicity holds (see Definition~\ref{3p8}), 
the preceding inequality is strict if $(\mbfm, \mbfu)\neq (m^N, u^N)$. 
In this case, $(\mbfm(t), \mbfu(t))$ is globally bounded and converges to $(m^N, u^N)$. Finally, 
this implies that $\Hh(t)$ converges to $\Hh^N$.

\section{Numerical results}
\label{impsec}

Here, we discuss the implementation of our numerical methods, the corresponding results, and some extensions.

In our numerical examples, we construct 
$F$ as follows. First, we build
\[
F^Q(p,q)=\frac 1 2 \left(\max\{p,q,0\}\right)^2
\]
and
\[
F^D(p,q,x)=
\begin{cases}
-b(x)p & \text{if } b(x)\leq 0,\\
b(x)q & \text{otherwise}.
\end{cases}
\]
We set
\[
F(p,q,x)=F^Q(p,q)+F^D(p,q,x)+V(x). 
\]
Then, $F_i$ is given by \eqref{Fi}.

We implemented our algorithms in MATLAB and Mathematica with no significant differences in performance or numerical results. We present here the computations performed with the Mathematica code. To solve the ordinary differential equations, we used the built-in Mathematica ODE solver with the stiff backward difference formula (BDF) discretization of variable order.

\subsection{Gradient flow}

For the gradient flow, we took $u(x,0)=0.2 \cos(2 \pi x)$ as the initial condition for $u$. We used $N=100$.
We set $b=0$ and $V(x)=\sin(2 \pi x)$.
 Figures \ref{u1dgrad} and~\ref{m1dgrad} feature the evolution of $u$ and \(m\), respectively, for $0\leq t\leq 1$. We can observe a fast convergence to the stationary solution $u=0$. Figure \ref{exactm1dgrad} illustrates
the behavior of $m$ at equally spaced times and compares it to the exact solution (in black).

        \begin{figure}[htb!]
                \begin{center}
                        \includegraphics[width=0.6\textwidth]{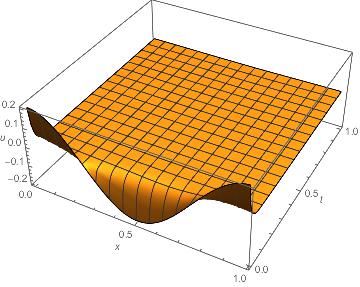}
                        \caption{Gradient
flow: $u$ }\label{u1dgrad}
                \end{center}
        \end{figure}

        \begin{figure}[htb!]
                \begin{center}
                        \includegraphics[width=0.6\textwidth]{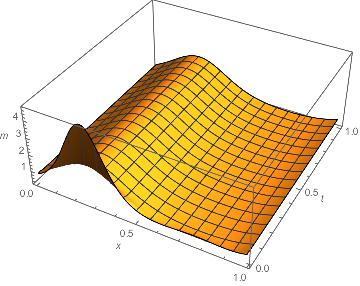}
                        \caption{Gradient
flow: $m$ }\label{m1dgrad}
                \end{center}
        \end{figure} 
        
        \begin{figure}[htb!]
                \begin{center}
                        \includegraphics[width=0.6\textwidth]{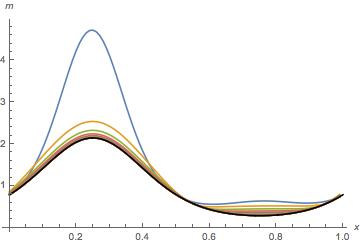}
                        \caption{Gradient flow:  $m$ numeric  for $0\leq t\leq 0.1$ vs.\! exact
(black line)}\label{exactm1dgrad}
                \end{center}
        \end{figure}

\subsection{Monotonic flow }

Here, we present the numerical results for the monotonic flow.
We set, as before, $b=0$ and $V(x)=\sin(2 \pi x)$.
We used $u(x,0)=0.2 \cos(2 \pi x)$ and 
$m(x,0)=1+0.2 \cos(2 \pi x)$  as initial conditions for $u$ and $m$, respectively.
As previously, we used $N=100$. Figures \ref{u1d} and \ref{m1d} depict the convergence to stationary solutions for $0\leq t\leq 10$. Figure \ref{exactm1d} compares the values of $m$ at equally spaced times with the stationary solution. Finally, Figures \ref{u1db} and \ref{m1db} illustrate the solution $(u,m)$ for the case $b=\cos^2(2 \pi x)$.

        \begin{figure}[htb!]
                \begin{center}
                        \includegraphics[width=0.6\textwidth]{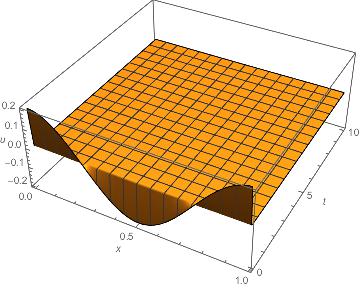}
                        \caption{Monotonic flow: $u$}\label{u1d}
                \end{center}
        \end{figure}

\begin{figure}[htb!]
        \begin{center}
                \includegraphics[width=0.6\textwidth]{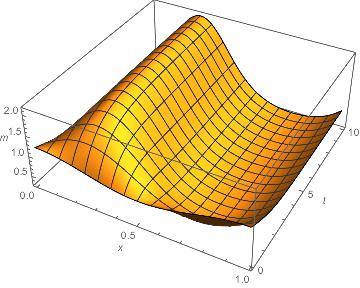}
                \caption{Monotonic flow: $m$}\label{m1d}
        \end{center}
\end{figure}

        \begin{figure}[htb!]
                \begin{center}
                        \includegraphics[width=0.6\textwidth]{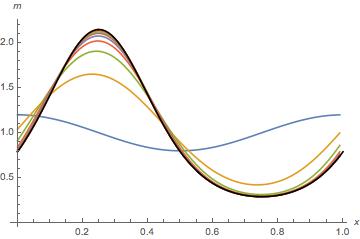}
                        \caption{Monotonic flow: $m$ numerical vs.\! exact (black line) for $t=0, 1,\hdots, 10$}\label{exactm1d}
                \end{center}
        \end{figure}

\begin{figure}[htb!]
        \begin{center}
                \includegraphics[width=0.6\textwidth]{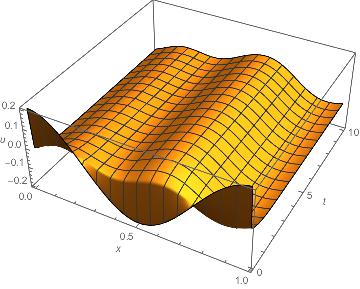}
                \caption{Monotonic flow: $u$ with  $b=\cos^2(2 \pi x)$}\label{u1db}
        \end{center}
\end{figure}

\begin{figure}[htb!]
        \begin{center}
                \includegraphics[width=0.6\textwidth]{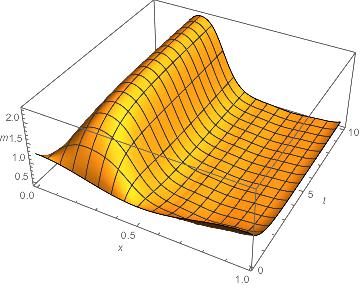}
                \caption{Monotonic flow: $m$  with  $b=\cos^2(2 \pi x)$}\label{m1db}
        \end{center}
\end{figure} 

\subsection{Application to congestion problems}

Our methods are not restricted to \eqref{1dsmfg} nor to one-dimensional problems. Here, we consider the congestion problem \eqref{1dsmfgcon} and present the corresponding numerical results. We examine higher-dimensional problems in the next section. 

The  congestion problem \eqref{1dsmfgcon} satisfies the monotonicity condition (see \cite{GMit}). Moreover, this problem admits the same 
explicit solution as \eqref{1dsmfg} with $b=0$. We chose $V(x)=\sin(2 \pi x)$, for comparison. 

We took the same
 initial conditions as in the previous section and set $N=100$. 
We present the evolution of $u$ and $m$ in Figures \ref{u1dcon} and \ref{m1dcon}, respectively. 
In Figure \ref{exactm1dcon}, we superimpose the exact solution, $m$, on the numerical values of $m$ at equally spaced times.

        \begin{figure}[htb!]
                \begin{center}
                        \includegraphics[width=0.6\textwidth]{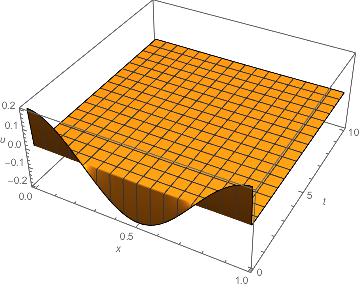}
                        \caption{Congestion model:
                         $u$}\label{u1dcon}
                \end{center}
        \end{figure}

        \begin{figure}[htb!]
                \begin{center}
                        \includegraphics[width=0.6\textwidth]{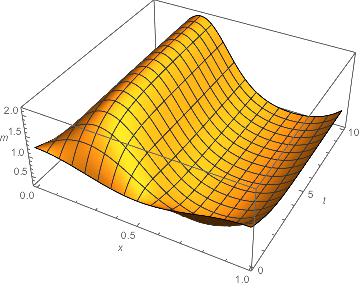}
                        \caption{Congestion model:
                        $m$}\label{m1dcon}
                \end{center}
        \end{figure} 

        \begin{figure}[htb!]
                \begin{center}
                        \includegraphics[width=0.6\textwidth]{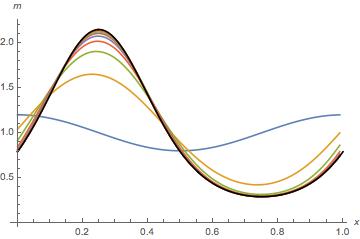}
                        \caption{Congestion model: $m$ numerical vs.\! exact (black line) for $t=0, 1,\hdots, 10$}\label{exactm1dcon}
                \end{center}
        \end{figure}

\subsection{Higher-dimensional examples}

The last example, we consider the following two-dimensional version of \eqref{1dsmfg}:
\begin{eqnarray}
\label{2d}
\begin{aligned}
\begin{cases}
\displaystyle \frac{w_x^2}{2}+ \frac{w_y^2}{2} + W(x,y) =\ln m+\Hh,\\
-(\theta (w_x))_x-(\theta (w_y))_y =0, 
\end{cases}
\end{aligned}
\end{eqnarray}
with $W(x,y)=\sin(2 \pi x)+\sin(2 \pi y)$. Because $W(x,y)=V(x)+V(y)$ for $V(x)=\sin(2 \pi x)$, the solution to \eqref{2d} takes the form $w(x,y)=u(x)+u(y)$
and $\theta(x,y)=m(x)+m(y)$ where $(u,m)$ solves \eqref{1dsmfg} with $b=0$.

We chose $w(x,y,0)=0.4 \cos(x+ 2 y)$, $\theta(x,y,0)=1+0.3 \cos(x- 3 y)$, and $N=20$. 
Figure~\ref{m2d} illustrates  $\theta$ at $T=50$. The numerical errors for   $\theta$ and $w$ are shown in Figures~\ref{m2derror} and \ref{u2derror},
respectively.

        \begin{figure}[htb!]
                \begin{center}
                        \includegraphics[width=0.6\textwidth]{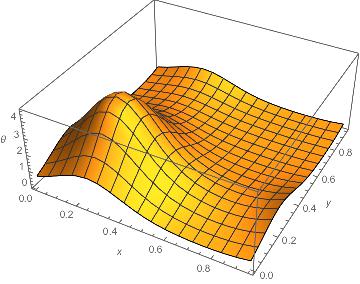}
                        \caption{Two-dimensional
                        problem: numeric $\theta$ 
at $T=50$}\label{m2d}
                \end{center}
        \end{figure}

        \begin{figure}[htb!]
                \begin{center}
                        \includegraphics[width=0.6\textwidth]{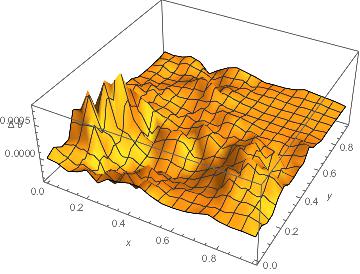}
                        \caption{Two-dimensional
                        problem: $\theta$ numerical error}\label{m2derror}
                \end{center}
        \end{figure}

        \begin{figure}[htb!]
                \begin{center}
                        \includegraphics[width=0.6\textwidth]{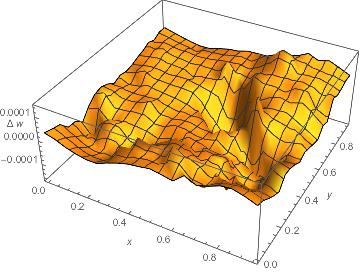}
                        \caption{Two-dimensional
                        problem: $w$ numerical error}\label{u2derror}
                \end{center}
        \end{figure}

\section{Final Remarks}
\label{conclusions}

Here, we developed two numerical methods to approximate solutions of stationary mean-field games. We addressed the convergence  of a discrete version of \eqref{1dsmfg}, and the convergence to weak solutions through a monotonicity argument. Our techniques generalize to discretized systems that are monotonic,  and that admit uniform bounds   with respect to the discretization parameter. 

In the cases we considered, our methods approximate well the exact solutions. While the gradient flow is considerably faster than the monotonic flow, this last  method applies to a wider class of problems.

We selected a simple  model for illustration purposes. In our numerical
examples, however, we illustrated the convergence of the schemes in higher-dimensional
problems and congestion MFG problems.
Furthermore, 
our results can be easily extended to related problems -- higher-dimensional cases, second-order MFG, or non-local (monotonic) problems.  Additionally, our methods provide a natural guide for two future research directions. The first is the development of a general theory of convergence for monotone schemes and the extension of our methods to mildly non-monotonic MFG. The second
is the study of time-dependent MFG.
This last direction is particularly relevant since the coupled structure of MFG and the initial-terminal conditions that are usually imposed make these problems very challenging from the numerical point of view.

\bibliographystyle{plain}

\bibliography{mfg}

\def\polhk#1{\setbox0=\hbox{#1}{\ooalign{\hidewidth
  \lower1.5ex\hbox{`}\hidewidth\crcr\unhbox0}}} \def\cprime{$'$}
\begin{thebibliography}{10}

\bibitem{achdou2013finite}
Y.~Achdou.
\newblock Finite difference methods for mean field games.
\newblock In {\em Hamilton-Jacobi Equations: Approximations, Numerical Analysis
  and Applications}, pages 1--47. Springer, 2013.

\bibitem{CDY}
Y.~Achdou, F.~Camilli, and I.~Capuzzo-Dolcetta.
\newblock Mean field games: numerical methods for the planning problem.
\newblock {\em SIAM J. Control Optim.}, 50(1):77--109, 2012.

\bibitem{DY}
Y.~Achdou and I.~Capuzzo-Dolcetta.
\newblock Mean field games: numerical methods.
\newblock {\em SIAM J. Numer. Anal.}, 48(3):1136--1162, 2010.

\bibitem{MR2928376}
Y.~Achdou and V.~Perez.
\newblock Iterative strategies for solving linearized discrete mean field games
  systems.
\newblock {\em Netw. Heterog. Media}, 7(2):197--217, 2012.

\bibitem{AP}
Y.~Achdou and A.~Porretta.
\newblock Convergence of a finite difference scheme to weak solutions of the
  system of partial differential equation arising in mean field games.
\newblock {\em Preprint}, 2015.

\bibitem{BS}
G.~Barles and P.~E. Souganidis.
\newblock Convergence of approximation schemes for fully nonlinear second order
  equations.
\newblock {\em Asymptotic Anal.}, 4(3):271--283, 1991.

\bibitem{bensoussan}
A.~Bensoussan, J.~Frehse, and P.~Yam.
\newblock {\em Mean field games and mean field type control theory}.
\newblock Springer Briefs in Mathematics. Springer, New York, 2013.

\bibitem{MR2928379}
F.~Camilli and F.~Silva.
\newblock A semi-discrete approximation for a first order mean field game
  problem.
\newblock {\em Netw. Heterog. Media}, 7(2):263--277, 2012.

\bibitem{MR3146865}
Fabio Camilli, Adriano Festa, and Dirk Schieborn.
\newblock An approximation scheme for a {H}amilton-{J}acobi equation defined on
  a network.
\newblock {\em Appl. Numer. Math.}, 73:33--47, 2013.

\bibitem{cardaliaguet}
P.~Cardaliaguet.
\newblock Notes on mean-field games.
\newblock 2011.

\bibitem{MR3358627}
P.~Cardaliaguet and P.~J. Graber.
\newblock Mean field games systems of first order.
\newblock {\em ESAIM Control Optim. Calc. Var.}, 21(3):690--722, 2015.

\bibitem{cgbt}
P.~Cardaliaguet, P.~J. Graber, A.~Porretta, and D.~Tonon.
\newblock Second order mean field games with degenerate diffusion and local
  coupling.
\newblock {\em NoDEA Nonlinear Differential Equations Appl.}, 22(5):1287--1317,
  2015.

\bibitem{MR3148086}
E.~Carlini and F.~J. Silva.
\newblock A fully discrete semi-{L}agrangian scheme for a first order mean
  field game problem.
\newblock {\em SIAM J. Numer. Anal.}, 52(1):45--67, 2014.

\bibitem{E6}
L.~C. Evans.
\newblock {\em Partial Differential Equations}.
\newblock Graduate Studies in Mathematics. American Mathematical Society, 1998.

\bibitem{E1}
L.~C. Evans.
\newblock Some new {PDE} methods for weak {KAM} theory.
\newblock {\em Calc. Var. Partial Differential Equations}, 17(2):159--177,
  2003.

\bibitem{E2}
L.~C. Evans.
\newblock Further {PDE} methods for weak {KAM} theory.
\newblock {\em Calc. Var. Partial Differential Equations}, 35(4):435--462,
  2009.

\bibitem{FG2}
R.~Ferreira and D.~Gomes.
\newblock Existence of weak solutions for stationary mean-field games through
  weak solutions.
\newblock {\em Preprint}.

\bibitem{GIMY}
D.~Gomes, R.~Iturriaga, H.~S{\'a}nchez-Morgado, and Y.~Yu.
\newblock Mather measures selected by an approximation scheme.
\newblock {\em Proc. Amer. Math. Soc.}, 138(10):3591--3601, 2010.

\bibitem{GPatVrt}
D.~Gomes, S.~Patrizi, and V.~Voskanyan.
\newblock On the existence of classical solutions for stationary extended mean
  field games.
\newblock {\em Nonlinear Anal.}, 99:49--79, 2014.

\bibitem{GPim1}
D.~Gomes and E.~Pimentel.
\newblock Local regularity for mean-field games in the whole space.
\newblock {\em To appear in Minimax Theory and its Applications}, 2015.

\bibitem{GPM3}
D.~Gomes, E.~Pimentel, and H.~S\'anchez-Morgado.
\newblock Time dependent mean-field games in the superquadratic case.
\newblock {\em To appear in ESAIM: Control, Optimisation and Calculus of
  Variations}.

\bibitem{GPM2}
D.~Gomes, E.~Pimentel, and H.~S{\'a}nchez-Morgado.
\newblock Time-dependent mean-field games in the subquadratic case.
\newblock {\em Comm. Partial Differential Equations}, 40(1):40--76, 2015.

\bibitem{GPM1}
D.~Gomes, G.~E. Pires, and H.~S{\'a}nchez-Morgado.
\newblock A-priori estimates for stationary mean-field games.
\newblock {\em Netw. Heterog. Media}, 7(2):303--314, 2012.

\bibitem{GM}
D.~Gomes and H.~S{\'a}nchez~Morgado.
\newblock A stochastic {E}vans-{A}ronsson problem.
\newblock {\em Trans. Amer. Math. Soc.}, 366(2):903--929, 2014.

\bibitem{GVW-dual}
D.~Gomes, R.~M. Velho, and M.-T. Wolfram.
\newblock Dual two-state mean-field games.
\newblock {\em Proceedings CDC 2014}, 2014.

\bibitem{Gomes:2014kq}
D.~Gomes, R.~M. Velho, and M.-T. Wolfram.
\newblock Socio-economic applications of finite state mean field games.
\newblock {\em Philos. Trans. R. Soc. Lond. Ser. A Math. Phys. Eng. Sci.},
  372(2028):20130405, 18, 2014.

\bibitem{GVrt2}
D.~Gomes and V.~Voskanyan.
\newblock Short-time existence of solutions for mean-field games with
  congestion.
\newblock {\em Preprint}.

\bibitem{GMit}
D.~A. Gomes and H.~Mitake.
\newblock Existence for stationary mean-field games with congestion and
  quadratic {H}amiltonians.
\newblock {\em NoDEA Nonlinear Differential Equations Appl.}, 22(6):1897--1910,
  2015.

\bibitem{GPat}
D.~A. Gomes and S.~Patrizi.
\newblock Obstacle mean-field game problem.
\newblock {\em Interfaces Free Bound.}, 17(1):55--68, 2015.

\bibitem{GPim2}
D.~A. Gomes and E.~Pimentel.
\newblock Time-{D}ependent {M}ean-{F}ield {G}ames with {L}ogarithmic
  {N}onlinearities.
\newblock {\em SIAM J. Math. Anal.}, 47(5):3798--3812, 2015.

\bibitem{MR3195844}
D.~A. Gomes and J.~Sa{\'u}de.
\newblock Mean field games models---a brief survey.
\newblock {\em Dyn. Games Appl.}, 4(2):110--154, 2014.

\bibitem{MR2974160}
O.~Gu{\'e}ant.
\newblock Mean field games equations with quadratic {H}amiltonian: a specific
  approach.
\newblock {\em Math. Models Methods Appl. Sci.}, 22(9):1250022, 37, 2012.

\bibitem{MR2976439}
O.~Gu{\'e}ant.
\newblock Mean field games with a quadratic {H}amiltonian: a constructive
  scheme.
\newblock In {\em Advances in dynamic games}, volume~12 of {\em Ann. Internat.
  Soc. Dynam. Games}, pages 229--241. Birkh\"auser/Springer, New York, 2012.

\bibitem{MR2928382}
O.~Gu{\'e}ant.
\newblock New numerical methods for mean field games with quadratic costs.
\newblock {\em Netw. Heterog. Media}, 7(2):315--336, 2012.

\bibitem{Caines2}
M.~Huang, P.~E. Caines, and R.~P. Malham{\'e}.
\newblock Large-population cost-coupled {LQG} problems with nonuniform agents:
  individual-mass behavior and decentralized {$\epsilon$}-{N}ash equilibria.
\newblock {\em IEEE Trans. Automat. Control}, 52(9):1560--1571, 2007.

\bibitem{Caines1}
M.~Huang, R.~P. Malham{\'e}, and P.~E. Caines.
\newblock Large population stochastic dynamic games: closed-loop
  {M}c{K}ean-{V}lasov systems and the {N}ash certainty equivalence principle.
\newblock {\em Commun. Inf. Syst.}, 6(3):221--251, 2006.

\bibitem{ll1}
J.-M. Lasry and P.-L. Lions.
\newblock Jeux \`a champ moyen. {I}. {L}e cas stationnaire.
\newblock {\em C. R. Math. Acad. Sci. Paris}, 343(9):619--625, 2006.

\bibitem{ll2}
J.-M. Lasry and P.-L. Lions.
\newblock Jeux \`a champ moyen. {II}. {H}orizon fini et contr\^ole optimal.
\newblock {\em C. R. Math. Acad. Sci. Paris}, 343(10):679--684, 2006.

\bibitem{ll3}
J.-M. Lasry and P.-L. Lions.
\newblock Mean field games.
\newblock {\em Jpn. J. Math.}, 2(1):229--260, 2007.

\bibitem{FrS}
A.~R. M{\'e}sz{\'a}ros and F.~J. Silva.
\newblock A variational approach to second order mean field games with density
  constraints: {T}he stationary case.
\newblock {\em J. Math. Pures Appl. (9)}, 104(6):1135--1159, 2015.

\bibitem{AO}
A.~M. Oberman.
\newblock Convergent difference schemes for degenerate elliptic and parabolic
  equations: {H}amilton-{J}acobi equations and free boundary problems.
\newblock {\em SIAM J. Numer. Anal.}, 44(2):879--895 (electronic), 2006.

\bibitem{PV15}
E.~Pimentel and V.~Voskanyan.
\newblock Regularity for second-order stationaty mean-field games.
\newblock {\em To appear in Indiana University Mathematics Journal}.

\bibitem{porretta}
A.~Porretta.
\newblock On the planning problem for the mean field games system.
\newblock {\em Dyn. Games Appl.}, 4(2):231--256, 2014.

\bibitem{porretta2}
A.~Porretta.
\newblock Weak solutions to {F}okker-{P}lanck equations and mean field games.
\newblock {\em Arch. Ration. Mech. Anal.}, 216(1):1--62, 2015.

\bibitem{San12}
F.~Santambrogio.
\newblock A modest proposal for {MFG} with density constraints.
\newblock {\em Netw. Heterog. Media}, 7(2):337--347, 2012.

\bibitem{MR3113415}
V.~K. Voskanyan.
\newblock Some estimates for stationary extended mean field games.
\newblock {\em Dokl. Nats. Akad. Nauk Armen.}, 113(1):30--36, 2013.

\end{thebibliography}

\end{document}